\documentclass[12pt,a4paper]{amsart}
\usepackage[top=35mm, bottom=30mm, left=30mm, right=30mm]{geometry}
\usepackage{mathptmx}
\usepackage{mathrsfs}
\usepackage{verbatim}
\usepackage{url}
\usepackage[all]{xy}
\usepackage{xcolor}
\usepackage[colorlinks=true,citecolor=blue]{hyperref}
\usepackage{amsmath,amsthm}
\usepackage{amssymb}
\usepackage{enumerate}
\usepackage{graphicx}

\newtheorem{thm}{Theorem}[section]
\newtheorem{cor}[thm]{Corollary}
\newtheorem{lem}[thm]{Lemma}
\newtheorem{prop}[thm]{Proposition}

\theoremstyle{definition}
\newtheorem{defin}[thm]{Definition}
\newtheorem{rem}[thm]{Remark}

\numberwithin{equation}{section}
\begin{document}


\baselineskip=17pt


\title{Some extensions of the Brouwer fixed point theorem}

\author[J.~Mai]{Jiehua Mai}
\address{School of Mathematics and Quantitative
Economics, Guangxi University of Finance and Economics, Nanning, Guangxi, 530003, P. R. China \&
 Institute of Mathematics, Shantou University, Shantou, Guangdong, 515063, P. R. China}
\email{jiehuamai@163.com; jhmai@stu.edu.cn}

\author[E.~Shi]{Enhui Shi}
\thanks{*Corresponding author}
\address{School of Mathematics and Sciences, Soochow University, Suzhou, Jiangsu 215006, China}
\email{ehshi@suda.edu.cn}

\author[K.~Yan]{Kesong Yan*}
\address{School of Mathematics and Statistics, Hainan Normal
University, Haikou, Hainan, 571158, P. R. China}
\email{ksyan@mail.ustc.edu.cn}

\author[F.~Zeng]{Fanping Zeng}
\address{School of Mathematics and Quantitative
Economics, Guangxi University of Finance and Economics, Nanning, Guangxi, 530003, P. R. China}
\email{fpzeng@gxu.edu.cn}

\begin{abstract}
We study the existence of fixed points for continuous maps $f$ from an $n$-ball $X$ in $\mathbb R^n$ to $\mathbb R^n$ with $n\geq 1$.
We show that $f$ has a fixed point if, for some absolute retract
$Y\subset\partial X$, $f(Y)\subset X$  and $\partial X-Y$ is an $(f, X)$-blockading set. For $n\geq 2$, let $D$ be an $n$-ball in $X$ and
$Y$ be an $(n-1)$-ball in $\partial X$. Relying on the result just mentioned,
we show the existence of a fixed point of $f$, if $D$ and $Y$ are well placed and behave well under  $f$, and ${\rm deg}(f_D)=-{\rm deg}(f_{\partial Y})$, where $f_D=f|D: D
\rightarrow \mathbb{R}^n$ and $f_{\partial Y}=f|\partial Y: \partial Y \rightarrow \partial Y$.
The degree ${\rm deg}(f_D)$ of $f_D$ is explicitly defined and some elementary properties of which
are investigated. These results extend the  Brouwer fixed point theorem.
\end{abstract}


\maketitle

\pagestyle{myheadings} \markboth{}{}

\section{Introduction}

In 1912, Brouwer obtain the
following fixed point theorem  by proving that homotopy mappings on a sphere have the same degree \cite{Bro12a}.
\begin{thm}[\cite{Bro12a}]\label{Brouwer}
Every continuous map $f$ from an $n$-ball $X$ to itself always has a fixed point\;\!.
\end{thm}
Nowadays, this elegant theorem has become one of the most well known
and useful principles in mathematics. It has a wide range of
applications in the fields of pure mathematics and applied
mathematics. Due to its importance in theory and application,
people's efforts to reprove or generalize this theorem have never
stopped for over a hundred years. For example, Knaster, Kuratowski
and Mazurkiewicz  gave a proof of the Brouwer fixed point theorem
based on Sperner's combinatorial lemma \cite{KKM29}. One may refer
to \cite{Hi76} and \cite{HW65} for the proofs based on differential
topology or algebraic topology. Milnor gave an analytic proof of
this theorem \cite{Mi78} and Rogers' proof only used advanced
calculus \cite{Rog80}. The Brouwer fixed point theorem was
generalized to continuous selfmaps of compact triangulable spaces
which has the same rational homology groups as a point \cite[p.
208]{Arm83} and to continuous selfmaps on compact convex subsets of
Banach spaces by Schauder \cite{Sch30} and of locally convex
topological vector spaces by Tychonoff \cite{Ty35}. Kakutani
generalized the theorem to multifunctions \cite{Ka41}, and
Glicksberg \cite{Gl52} and  Fan \cite{Fan52} further generalized
Kakutani's result to locally convex topological vector spaces. One
may consult \cite{Park99}  for a detailed introduction to all kinds
of equivalent formulations and extensions of the Brouwer fixed point
theorem. Some recent proofs and extensions of Theorem \ref{Brouwer}
can be found in \cite{ Bi04, Ma05, Mi02, Sub18}.
\medskip

 The following generalization
is mentioned in \cite{HL68}, which is an immediate consequence of
Brouwer's fixed point theorem and Bing's retraction \cite{Bi64}.
\begin{thm}[\cite{HL68}]\label{boundary}
Let $X$ be an $n$-ball in $\mathbb R^n$ and $f:X\rightarrow \mathbb R^n$ be continuous. If
$f(\partial X)\subset X$, then $f$ has a fixed point in $X$.
\end{thm}
Here an {\it $n$-ball} $X$ in $\mathbb R^n$ means that $X$ is an embedding of the unit closed ball
$B^n$ of $\mathbb R^n$. We should note that the embedding $X$ of $B^n$ may be very wild as the Alexander horned sphere illustrated \cite[p.176]{HY61}.
\medskip

The aim of the paper is to extend Theorem \ref{boundary} to the case when $f(\partial X)\not\subset X$; that is, we try to find some
conditions which ensure the existence of a fixed point of $f$ even if $f$ only maps a part of $\partial X$ into $X$.
\medskip

In Section 3, we introduce the notion of blockading set.
Intuitively, for a  ball $X\subset\mathbb R^n$ and a continuous map
$f:X\rightarrow \mathbb R^n$, an $(f, X)$-blockading set $Y$ of
$\partial X$ will forbid some points of $X$ from passing across
$\partial X$, but points of $Y$ may be mapped outside $X$ by $f$
(see Definition \ref{block} for the details). Using this notion, we
obtain the following theorem.

\begin{thm} \label {main1} Let $X$ be an $n$-ball
in $\Bbb R^n$, $n \geq 1$, $Y\subset \partial X$ be an absolute
retract, and $f: X \rightarrow \mathbb{R}^n$ be a continuous map. If
$f(Y) \subset X$, and $\partial X-Y$ is an $(f, X)$-blockading set,
then $f$ has a fixed point in $X$.
\end{thm}

Noting that $\partial X-Y$ is always an $(f, X)$-blockading set provided that $f(X)\cap (\partial X-Y)=\emptyset$,
the following corollary is immediate.

\begin{cor} Let $X$ be an
$n$-ball in $\Bbb R^n$, $n \geq 1$, $Y\subset \partial X$ be an
absolute retract, and $f: X \rightarrow \mathbb{R}^n$ be a
continuous map. If $f(Y) \subset X$, and $f(X) \cap (\partial
X-Y)=\emptyset$, then $f$ has a fixed point.
\end{cor}

In Section 2, for an injective map $f$ from an $n$-ball $D$ to
$\mathbb R^n$  which contains $D$ or to another $n$-ball $Y$ which
contains $D$ and is contained in a topological space, we introduce
the notion of degree of $f$, written by ${\rm deg}(f)$. This is a
supplement of several classical definitions of degree of a map. Some
elementary properties of this notion are investigated. Specially, it
can be shown that ${\rm deg}(f)\in\{1, -1\}$ for the definitions
here.
\medskip

In Section 4, relying on Theorem \ref{main1} and the notions of degree mentioned above, we obtain the following results.

\begin{thm} \label{main2} Let $X$ be an $n$-ball
in $\mathbb{R}^n$, $n \geq 2$, $Y$ be an $(n-1)$-ball in $\partial
X$, and $D$ be an $n$-ball in $X$ such that $E\equiv
\partial D \cap \partial X$ is an $(n-1)$-ball in
$\stackrel{\circ}Y$. Let $f: X \rightarrow \mathbb{R}^n$ be a
continuous map such that $f$ is bijective on $D$, and
$f^{-1}(\partial X-Y)=\stackrel{\circ}Y$. Let $f_D=f|D: D
\rightarrow \mathbb{R}^n$ and $f_{\partial Y}=f|\partial Y: \partial
Y \rightarrow \partial Y$. If
$\mathrm{deg}(f_D)=-\mathrm{deg}(f_{\partial Y})$,  then $f$ has a fixed point.
\end{thm}

The following theorem is a corollary of Theorem \ref{main2}.

\begin{thm} Let $X$ be an $n$-ball
in $\mathbb{R}^n$, $n \geq 2$, $Y$ be an $(n-1)$-ball in $\partial
X$, and $f: X \rightarrow \mathbb{R}^n$ be a continuous injection
such that $f(Y)=\partial X-\stackrel{\circ}Y$. Let $f_{\partial
Y}=f|\partial Y: \partial Y \rightarrow \partial Y$. If
$\mathrm{deg}(f)=-\mathrm{deg}(f_{\partial Y})$, then $f$ has a
fixed point.
\end{thm}

\section{Degree of a map}
In this section, we will recall the classical definitions of degree of a map.
Then for an injective map $f$ from an $n$-ball $D$ to $\mathbb R^n$ containing $D$ or to
another ball $Y$ which contains $D$ and is contained in a topological space, we introduce the notion of degree of $f$ and investigate
some of their elementary properties. These will be used in establishing the main results.
\medskip

\subsection{Classical definitions of degree of a map}
For any $n\in\Bbb N$\;\!, \,let $\Bbb R^n$ be the
$n$\;\!-dimensional Euclidean space, let $\mathbf{0}$ denote the
origin of $\Bbb R^n$, and let $d_E$ be the Euclidean metric on $\Bbb
R^n$. For any $\mathbf{x} \in \Bbb R^n$, write
$\|\mathbf{x}\|=d_E(\mathbf{x},\textbf{0})$. For any $\mathbf{x} \in
\Bbb R^n$ and any $r>0$, write $B^n(\mathbf{x},r)=\{\mathbf{y} \in
\Bbb R^n:\|\mathbf{y}-\mathbf{x}\|\leq r\}$. Let
$B^n=B^n(\mathbf{0}, 1)$ and $S^{n-1}=\{\mathbf{y} \in \Bbb
R^n:\|\mathbf{y}\|=1\}$. The sets $B^n$ and $S^{n-1}$ are called the
{\it\textbf{unit $n$-ball}} \,and the {\it\textbf{unit
$(n\!-\!1)$-sphere}}\;\!,\, respectively. Note that $S^0=\{-1,1\}$
contains only two points. Each space homeomorphic to $B^n$
($B^n\!-S^{\,n-1}$, $S^{\,n-1}$, respectively) is called an
{\it\textbf{$n$-ball}} ({\it\textbf{open $n$-ball}},
{\it\textbf{$(n-1)$-sphere}}, respectively)\;\!. Each $1$-\;\!ball
is called an {\it\textbf{arc}}\;\!, \,each $1$-\;\!sphere is called
a {\it\textbf{circle}}\;\!, \,and each $2$\;\!-ball is called a
{\it\textbf{disk}}\;\!. Usually, the points $(r_1,\cdots,r_n)\in\Bbb
R^n$ and $(r_1,\cdots,r_n,0)\in\Bbb R^{n+1}$ are regarded to be the
same. This means that\;\! $\Bbb R^n=\Bbb R^n\times\{0\}\subset\Bbb
R^{n+1}$, \,$B^n=B^n\times\{0\}\subset B^{n+1}$ \,and
$S^{n-1}=S^{n-1}\times\{0\}\subset S^{n}$. For any $n$-ball $X$ and
any homeomorphism $h:B^n\to X$,\, write \,$\partial X=h(S^{\,n-1})$
\,and \,$\stackrel{\circ}X=X-\partial X$, \,called the
{\it\textbf{boundary}} \,and the {\it\textbf{interior}} \,of \,$X$ ,
respectively. Note that \,$\partial X$\;\! and \,$\stackrel{\circ}X$
are independent of the choice of the homeomorphism $h:B^n\to X$.
\medskip

Recall that if $\widetilde{H}_n(S)$ is the $n$-th reduced (singular) homology group
of an $n$-sphere $S$, then  $\widetilde{H}_n(S)\cong \mathbb Z$;
if $H_n(Y,\,\partial Y)$ is the $n$-th relative (singular) homology
group of the $n$-ball $Y$, then $H_n(Y,\,\partial Y)\cong \mathbb Z$.

\medskip
The following definition can be seen in \cite[p.288, p.293]{Mau80}.

\begin{defin}\label{class-deg}(1) Let $S$ be a $n$-sphere, $n \geq 0$,
$\widetilde{H}_n(S)$ be the $n$-th reduced (singular) homology group
of $S$ , and $f: S \rightarrow S$ be a continuous map. Then $f$
induces a homomorphism $f_*: \widetilde{H}_n(S) \rightarrow
\widetilde{H}_n(S)$, and there is an integer $\lambda_f$ such that
$f_*(\alpha)=\lambda_f \alpha$ for any $\alpha \in
\widetilde{H}_n(S)$. This integer $\lambda_f$ is called the {\it
\textbf{degree}} of the map $f$ and we write
$\mathrm{deg}(f)=\lambda_f$.

\vspace{3mm} (2) Let $Y$ be an $n$-ball, $n \geq 1$,
$H_n(Y,\,\partial Y)$ be the $n$-th relative (singular) homology
group of $(Y,\,\partial Y)$, and $f: (Y,\,\partial Y) \rightarrow
(Y,\,\partial Y)$ be a continuous map. Then $f$ induces a
homomorphism $f_*: H_n(Y,\,\partial Y) \rightarrow H_n(Y,\,\partial
Y)$, and there is an integer $\lambda_f$ such that
$f_*(\alpha)=\lambda_f \alpha$ for any $\alpha \in H_n(Y,
\partial Y)$. This integer is called the {\it \textbf{degree}} of
$f$ and we write $\mathrm{deg}(f)=\lambda_f$.

\medskip The following five lemmas are well known (see e.g.
\cite{Mau80}).

\begin{lem} \label{2.2} Let $S$ be a $n$-sphere,
$n \geq 0$, and $f, g: S \rightarrow S$ be continuous maps. Then
$\mathrm{deg}(f)=\mathrm{deg}(g)$ if and only if $f$ and $g$ are
homotopic; ${\rm deg}(f\circ g)={\rm deg}(f)\cdot {\rm deg}(g)$.
\end{lem}

\begin{lem} \label{deg-bound} Let $Y$ be an $n$-ball,
$n \geq 1$, and $f: (Y, \partial Y) \rightarrow (Y, \partial Y)$ be
a continuous map. Then $\mathrm{deg}(f)=\mathrm{deg}(f|\partial
Y)$.
\end{lem}

\begin{lem} Let $Y$ be an $n$-ball, $n
\geq 1$, and $f, g: (Y, \partial Y) \rightarrow (Y, \partial Y)$ be
continuous maps. Then $\mathrm{deg}(f)=\mathrm{deg}(g)$ if and only
if $f$ and $g$ are homotopic (that is, there exists a homotopy $F:(Y
\times [0,1], \partial Y \times [0,1]) \rightarrow (Y,
\partial Y)$ from $f$ to $g$).
\end{lem}

\begin{lem} Let $S$ and $T$ be two
$n$-spheres, $n \geq 0$, and $f: S \rightarrow S$ and $g: T
\rightarrow T$ be two continuous maps. If there exists a
homeomorphism $h: S \rightarrow T$ such that $g \circ h=h \circ f$,
then $\mathrm{deg}(f)=\mathrm{deg}(g)$.
\end{lem}

\begin{lem}\label{conj-deg} Let $V$ and $W$ be two
$n$-balls, $n \geq 1$, and $f: (V, \partial V) \rightarrow (V,
\partial V)$ and $g: (W, \partial W) \rightarrow (W, \partial W)$ be two continuous maps. If there
exists a homeomorphism $h: V \rightarrow W$ such that $g \circ h=h
\circ f$, then $\mathrm{deg}(f)=\mathrm{deg}(g)$.
\end{lem}

\subsection{Extended definitions of degree of a map}

Let $f: \Bbb R^n \rightarrow \Bbb R^n$ be a continuous map, $n \geq
1$. Write $\Bbb R_{\,\infty}^{\,n}=\Bbb R^n \cup \{\infty\}$. Then
$\Bbb R_{\,\infty}^{\,n}$ is an $n$-ball. If $\lim_{\|x\|\rightarrow
\infty}f(x)=\infty$, then we say that $f$ is {\it
\textbf{$\infty$-extensible}}, and we can define a continuous map
$\widehat{f}: \Bbb R_{\,\infty}^{\,n} \rightarrow \Bbb
R_{\,\infty}^{\,n}$ by $\widehat{f}\,|\Bbb R^n=f$ and
$\widehat{f}(\infty)=\infty$, which is called the {\it
\textbf{natural extension}} of $f$.

\begin{defin}\label{inf-ext}
 Let $f: \Bbb R^n \rightarrow \Bbb R^n$ be an $\infty$-extensible
map with $n \geq 1$ and $\widehat{f}(\infty)=\infty$ be its natural extension.
Then we call the number $\mathrm{deg}(\widehat{f})$  the
{\it \textbf{degree}} of $f$, which is written by $\mathrm{deg}(f)$.
\end{defin}

 Let $f$ be a homeomorphism of an $n$-sphere $S$
with $n \geq 0$, or of an $n$-ball $Y$, or of $\mathbb{R}^n$ with $n
\geq 1$. If $\mathrm{deg}(f)=1$ (resp. $\mathrm{deg}(f)=-1$), then
$f$ is said to be {\it \textbf{orientation preserving}} (resp.
{\it\textbf{orientation reversing}}).
\end{defin}

\begin{rem}
Note that if $A$ is an
arc and $\partial A=\{u,v\}$, then $A$ is a $1$-ball, $\{u,v\}$ is a
$0$-sphere, and a bijection\, $f:\{u,v\}\to\{u,v\}$ is orientation
preserving (resp. orientation reversing) if and only if $f(u)=u$
(resp. $f(u)=v$)\;\!.
\end{rem}

\begin{defin} \label{bij} Let $X$ and $Y$ be two
topological spaces, $V \subset X$, and $f: X \rightarrow Y$ be a
map. If $f|V$ is injective, and $f^{-1}(f(V))=V$, then we say that
$f$ is {\it \textbf{bijective on}} $V$.
\end{defin}

Note that, in Definition \ref{bij}, if $f:X\rightarrow Y$ itself is injective,
then for any $V\subset X$, $f$ is always bijective on $V$.

\begin{lem}\label{equal} Let $f: \Bbb R^n
\rightarrow \Bbb R^n$ and $g: \Bbb R^n \rightarrow \Bbb R^n$ be two
$\infty$-extensible continuous maps, $n \geq 1$. If there exists an
$n$-ball $V$ in $\Bbb R^n$ such that $f|V=g|V$, and both $f$ and $g$
are bijective on $V$, then

\vspace{3mm} $(1)$\ $\mathrm{deg}(f)=\mathrm{deg}(g)$;

\vspace{3mm} $(2)$\ $\mathrm{deg}(f) \in \{1, -1\}$.
\end{lem}

\begin{proof} (1) Let $\widehat{f}: \Bbb
R_{\;\infty}^{\;n} \rightarrow \Bbb R_{\;\infty}^{\;n}$ and
$\widehat{g}: \Bbb R_{\;\infty}^{\;n} \rightarrow \Bbb
R_{\;\infty}^{\;n}$ be the natural extensions of $f$ and $g$,
respectively. Write $W=f(V)$. Take $w \in \stackrel{\circ}W$ and
$r>0$ such that $B^n(w,r) \subset W$. Let $h: \Bbb
R_{\;\infty}^{\;n}-{\stackrel{\circ}{~B^{\;n}}}(w,r) \rightarrow
B^n$ be a homeomorphism. Write $X=f^{-1}\big(B^n(w,r)\big)$. Define
$F: \Bbb R_{\;\infty}^{\;n} \times [0,1] \rightarrow \Bbb
R_{\;\infty}^{\;n}$ by, for any $x \in \Bbb R_{\;\infty}^{\;n}$ and
any $t \in [0,1]$,
$$F(x,t)=\left\{\begin{array}{cl}\widehat{f}(x)=\widehat{g}(x), & \mbox{\ if\ } x \in X \cup \{\infty\};\\
h^{-1}\big((1-t) \cdot h \circ f(x)+t \cdot h \circ g(x)\big), &
\mbox{\ if\ } x \in \Bbb R^n-X.\end{array}\right.$$ Then $F$ is a
homotopy from $\widehat{f}$ to $\widehat{g}$. By Lemma \ref{2.2}, we have
$\mathrm{deg}(f)=\mathrm{deg}(\widehat{f})=\mathrm{deg}(\widehat{g})=\mathrm{deg}(g)$.

\vspace{3mm} (2) Write $v=f^{-1}(w)$. Take $s>0$ such that
$f\big(B^n(v,s)\big) \subset B^n(w,r)$. Define $\varphi: \Bbb R^n
\rightarrow \Bbb R^n$ by $\varphi(y)=f^{-1}(y)$ for any $y \in
B^n(w,r)$ and $\varphi(w+t(z-w))=v+ t(\varphi(z)-v)$ for any $z \in
\partial B^n(w,r)$ and any $t \geq 1$. Then $\varphi$ is an
$\infty$-extensible continuous map. Let $\psi=\varphi \circ f: \Bbb
R^n \rightarrow \Bbb R^n$. Then $\psi$ is also an
$\infty$-extensible continuous map, $\psi|B^n(v,s)=id$, and $\psi$
is bijective on $B^n(v,s)$. Noting that $id_{\Bbb R^n}$ is also
bijective on $B^n(v,s)$, from Lemma 2.2 and the proved conclusion (1) of this
lemma we get \vspace{-1mm}
$$\mathrm{deg}(\varphi) \cdot \mathrm{deg}(f)=\mathrm{deg}(\varphi \circ f)=\mathrm{deg}(\psi)
=\mathrm{deg}(id_{\Bbb R^n})=1.$$\vspace{-1mm} Hence,
$\mathrm{deg}(f)=\mathrm{deg}(\varphi)\in\{1, -1\}$. \end{proof}

\begin{defin} Let $V$ be an $n$-ball in
$\Bbb R^n$, $n \geq 1$, $W \subset \Bbb R^n$, $h: V \rightarrow W$
be a continuous injection, and $f: \Bbb R^n \rightarrow \Bbb R^n$ be
a continuous map. If there exists an $n$-ball $D \subset V$ such
that $f|D=h|D$, then $f$ is called a {\it \textbf{pseudo-extension}}
of $h$. Further, if $f$ is $\infty$-extensible, and  $f$
is bijective on $D$, then $f$ is called an {\it
\textbf{$\infty$-extensible locally bijective pseudo-extension}} of
$h$.
\end{defin}

\begin{lem} \label{existence} Let $V$ be an $n$-ball
in $\Bbb R^n$, $n \geq 1$, $W \subset \Bbb R^n$ and $h: V
\rightarrow W$ be a continuous injection. Then

\vspace{3mm} $(1)$\ $h$ has an $\infty$-extensible locally bijective
pseudo-extension.

\vspace{3mm} $(2)$\ For any given $\{v,u\} \subset
\stackrel{\circ}V$, there exist an $n$-ball $D \subset V$ and an
$\infty$-extensible continuous map $f: \Bbb R^n \rightarrow \Bbb
R^n$ such that $\{v,u\} \subset \stackrel{\circ}D$, $f|D=h|D$ and $f$ is
bijective on $D$.
\end{lem}

\begin{proof} Clearly, (2) implies (1). Thus it
suffices to prove (2). Take an $r>0$ such that $B^n(v, r) \subset
V$. Write $w=h(v)$. Define $g: \Bbb R^n \rightarrow \Bbb R^n$ by
$g|B^n(v,r)=h|B^n(v,r)$ and
$$g\big(v+t(z-v)\big)=w+ t\big(h(z)-w\big)$$
for any $z \in \partial B^n(v,r)$ and any $t \geq 1$. Then $g$ is an
$\infty$-extensible continuous maps. Take $s>0$ such that $B^n(w,s)
\subset h\big(B^n(v,r)\big)$. Write $E=h^{-1}\big(B^n(w,s)\big)$.
Then $E$ is an $n$-ball, $E \subset B^n(v,r)$, $g|E=h|E$, and $g$ is
bijective on $E$. Take a homeomorphism $\eta: \Bbb R^n \rightarrow
\Bbb R^n$ such that $\eta|(\Bbb R^n-\stackrel{\circ}V )=id$ and
$\eta(\stackrel{\circ}E) \supset \{v,u\}$. Write $D=\eta(E)$. Then
$D$ is an $n$-ball, and $\{v,u\} \subset \stackrel{\circ}D$. Take a
homeomorphism $\xi: \Bbb R^n \rightarrow \Bbb R^n$ such that
$\xi|W=h \circ \eta \circ h^{-1}$ and $\xi|(\Bbb
R^n-\stackrel{\circ}W)=id$. Let $f=\xi \circ g \circ \eta^{-1}: \Bbb
R^n \rightarrow \Bbb R^n$. Then $f$ is an $\infty$-extensible
continuous map. Consider any $y \in \Bbb R^n$. If $y \in D$, then
$\eta^{-1}(y) \in E$, which implies $g \circ \eta^{-1}(y)=h \circ
\eta^{-1}(y) \in h(E) \subset W$, and hence
$$f(y)=\xi \circ g \circ \eta^{-1}(y)=h \circ \eta \circ h^{-1} \circ h \circ \eta^{-1}(y)=h(y).$$
Thus we get $f|D=h|D$. If $y \notin D$, then $\eta^{-1}(y) \notin
E$, which implies $g \circ \eta^{-1}(y) \notin g(E)$, and hence
\begin{eqnarray*}
f(y) &=& \xi \circ g \circ \eta^{-1}(y) \in \xi\left(\Bbb
R^n-W\right) \cup \xi\big(W-g(E)\big)\\
&=& \left(\Bbb R^n-W\right) \cup \big(W-\xi \circ g(E)\big)=\Bbb
R^n-\xi \circ  h(E)\\
&=& \Bbb R^n-h \circ \eta \circ h^{-1} \circ h(E)=\Bbb R^n-h(D)=\Bbb
R^n-f(D).
\end{eqnarray*}
Thus $f$ is bijective on $D$.
\end{proof}

\begin{lem} \label{equality} Let $V$ be an $n$-ball in
$\Bbb R^n$, $n \geq 1$, $W \subset \Bbb R^n$, and $h: V \rightarrow
W$ be a continuous injection. Then any two $\infty$-extensible
locally bijective pseudo-extensions of $h$ have the same degree $1$
or $-1$.
\end{lem}

\begin{proof} Let $f_1: \Bbb R^n \rightarrow
\Bbb R^n$ and $f_2: \Bbb R^n \rightarrow \Bbb R^n$ be two
$\infty$-extensible locally bijective pseudo-extensions of $h$. Then
for $j=1, 2$, there exists an $n$-ball $D_j \subset V$ such that
$f_j|D_j=h|D_j$, and $f_j$ is bijective on $D_j$. Take a point $v_j
\in \stackrel{\circ}{D_j}$. By Lemma \ref{existence}, there exist an $n$-ball
$D \subset V$ and an $\infty$-extensible continuous map $f: \Bbb R^n
\rightarrow \Bbb R^n$ such that $\{v_1, v_2\} \subset
\stackrel{\circ}D$, $f|D=h|D$, and $f$ is bijective on this $n$-ball
$D$. Take $\varepsilon>0$ such that $B^n(v_j, \varepsilon) \subset D
\cap D_j$ for $j=1,2$. From Lemma \ref{equal}, we get
$\mathrm{deg}(f_1)=\mathrm{deg}(f)=\mathrm{deg}(f_2) \in \{1, -1\}$.
\end{proof}

Lemma \ref{existence}-(1) and Lemma \ref{equality} make the following definition reasonable.

\begin{defin}\label{degree-ext}Let $V$ be an $n$-ball in
$\Bbb R^n$, $n \geq 1$, $W \subset \Bbb R^n$, and $h: V \rightarrow
W$ be a continuous injection. Take an $\infty$-extensible locally
bijective pseudo-extension $f$ of $h$. Define the {\it
\textbf{degree}} of $h$ to be $\mathrm{deg}(h)=\mathrm{deg}(f)$, and
if $\mathrm{deg}(h)=1$ (resp. $\mathrm{deg}(h)=-1$) then we say that
$h$ is {\it \textbf{orientation preserving}} (resp. {\it
\textbf{orientation reversing}}).
\end{defin}

\begin{rem}\label{rem1} $(1)$\ Let $W' \subset \Bbb R^n$ and let $h': V
\rightarrow W'$ be a continuous injection such that $h'(v)=h(v)$ for
any $v \in V$. Then by Definition \ref{degree-ext} we have
$\mathrm{deg}(h')=\mathrm{deg}(f)=\mathrm{deg}(h)$. Thus, when we
discuss degrees, the maps $h'$ and $h$ can be regard as the same.

\vspace{3mm} $(2)$ Let $X$ be an $n$-ball in $\Bbb R^n$, $n \geq 2$,
and $Y$ be an $(n-1)$-ball in $\partial X$. Then

\vspace{3mm} $(a)$ $X$ is said to be {\it\textbf{benign}} in
$\mathbb{R}^n$ if $\Bbb R_{\;\infty}^{\;n}-\stackrel{\circ}X$ is
also an $n$-ball.

\vspace{3mm} $(b)$ $Y$ is said to be {\it\textbf{benign}} in
$\partial X$ if $\partial X-\stackrel{\circ}Y$ is also an
$(n-1)$-ball.

\vspace{3mm} \noindent It is easy to see that the following three
conditions are equivalent:

\vspace{3mm} $(c)$ $X$ is benign $n$-ball in $\mathbb{R}^n$;

\vspace{3mm} $(d)$ There exists a homeomorphism $\eta_0: B^n
\rightarrow X$ which can be extended to a homeomorphism $\eta: \Bbb R^n
\rightarrow \Bbb R^n$;

\vspace{3mm} $(e)$ Every homeomorphism $\eta_0: B^n \rightarrow X$
can be extended to a homeomorphism $\eta: \Bbb R^n \rightarrow \Bbb R^n$.

\vspace{3mm}\noindent In Definition \ref{degree-ext}, if $h: V \rightarrow W$ is a
homeomorphism and $h$ can be extended to a homeomorphism $\eta: \Bbb R^n
\rightarrow \Bbb R^n$, then we can define the degree of $h$ to be
$\mathrm{deg}(h)=\mathrm{deg}(\eta)$. However, for $n \geq 3$, we
know there exist $n$-balls in $\Bbb R^n$ which are not benign. For
example, if $W$ is a $3$-ball in $\Bbb R^3$ and $\partial W$ is an
Alexander horned sphere, then $W$ is not benign \cite[p.
385]{Mun75}. Thus it is possible that the homeomorphism $h: V
\rightarrow W$ cannot be extended to a homeomorphism of $\Bbb R^n$.
Luckily, even if $h: V \rightarrow W$ cannot be extended to a
homeomorphism of $\Bbb R^n$, we can still choose an
$\infty$-extensible locally bijective pseudo-extension $f$ of $h$
and define the degree of $h$ as $\mathrm{deg}(h)=\mathrm{deg}(f)$.

\end{rem}

\begin{lem}\label{restrict} Let $V' \subset V$ be two
$n$-balls in $\Bbb R^n$, $n \geq 1$, $W \subset \Bbb R^n$, and let
$h: V \rightarrow W$ be a continuous injection. Let $h'=h|V': V'
\rightarrow W$. Then $\mathrm{deg}(h')=\mathrm{deg}(h)$.
\end{lem}

\begin{proof} Let $f: \Bbb R^n \rightarrow \Bbb
R^n$ be an $\infty$-extensible locally bijective pseudo-extensions
of $h'$. Then $f$ is also an $\infty$-extensible locally bijective
pseudo-extension of $h$. Thus we have
$\mathrm{deg}(h')=\mathrm{deg}(f)=\mathrm{deg}(h)$.
\end{proof}

\begin{lem} \label{composition} Let $V, W$ and $X$ be
$n$-balls in $\Bbb R^n$, and let $h: V \rightarrow W$ and $\eta: W
\rightarrow X$ be two continuous injections. Then $\mathrm{deg}(\eta
\circ h)=\mathrm{deg}(\eta) \cdot \mathrm{deg}(h)$.
\end{lem}

\begin{proof} Let $f: \Bbb R^n \rightarrow \Bbb
R^n$ be an $\infty$-extensible locally bijective pseudo-extensions
of $h$. Then there is an $n$-ball $D \subset V$ such that $f|D=h|D$,
and $f$ is bijective on $D$. By Lemma \ref{existence}, there exist an $n$-ball
$E \subset f(D)$ and an $\infty$-extensible continuous map $g: \Bbb
R^n \rightarrow \Bbb R^n$ such that $g|E=\eta|E$, and $g$ is
bijective on $E$. Clearly, $g \circ f$ is an $\infty$-extensible
locally bijective pseudo-extensions of $\eta \circ h$. Hence, by Lemma \ref{2.2},
$\mathrm{deg}(\eta \circ h)=\mathrm{deg}(g \circ f)=\mathrm{deg}(g)
\cdot \mathrm{deg}(f)=\mathrm{deg}(\eta) \cdot \mathrm{deg}(h)$.
\end{proof}

\begin{lem}\label{inverse} Let $V$ and $W$ be
$n$-balls in $\Bbb R^n$, and $h: V \rightarrow W$ be a
homeomorphism. Then $\mathrm{deg}(h^{-1})=\mathrm{deg}(h) \in \{1,
-1\}$.
\end{lem}

\begin{proof} By Lemma \ref{composition}, we have
$\mathrm{deg}(h^{-1}) \cdot \mathrm{deg}(h)=\mathrm{deg}(h^{-1}\circ
h)=\mathrm{deg}(id|V)=1$. Then the conclusion follows from (2) of Lemma \ref{equal}.
\end{proof}

\begin{lem} \label{deg-conj} Let $V, W, X$ and $X'$
be $n$-balls in $\Bbb R^n$ with $V \cup W \subset X$, $n \geq 1$.
Let $h: V \rightarrow W$ be a continuous injection, and $\eta: X
\rightarrow X'$ be a homeomorphism. Write $V'=\eta(V)$ and
$W'=\eta(W)$. Put $h'=\eta \circ h \circ \eta^{-1}|V': V'
\rightarrow W'$. Then $\mathrm{deg}(h')=\mathrm{deg}(h)$.
\end{lem}

\begin{proof} By Lemmas \ref{composition}, \ref{restrict} and \ref{inverse}, we
have $\mathrm{deg}(h')=\mathrm{deg}(\eta|W) \cdot \mathrm{deg}(h)
\cdot \mathrm{deg}(\eta^{-1}|V')=\mathrm{deg}(\eta) \cdot
\mathrm{deg}(h) \cdot \mathrm{deg}(\eta^{-1})=\mathrm{deg}(h)$. Thus
$\mathrm{deg}(h')=\mathrm{deg}(h)$.
\end{proof}

\begin{lem}\label{uniform} Let $f: \Bbb R^{n}
\rightarrow \Bbb R^n$ be an $\infty$-extensible continuous map, $n
\geq 1$, and $V$ be an $n$-ball in $\Bbb R^n$. Suppose that $f$ is
bijective on $V$. Write $h=f|V: V \rightarrow \Bbb R^n$. Let
$\mathrm{deg}(f)$ be defined as in Definition \ref{inf-ext}, and let
$\mathrm{deg}(h)$ be defined as in Definition \ref{degree-ext}. Then

\vspace{3mm} $(1)$\ $\mathrm{deg}(f)=\mathrm{deg}(h) \in \{1, -1\}$;

\vspace{3mm} $(2)$\ If $W$ is also an $n$-ball in $\mathbb{R}^n$,
$f$ is also bijective on $W$, and put $g=f|W: W \rightarrow
\mathbb{R}^n$, then $\mathrm{deg}(g)=\mathrm{deg}(h)$.
\end{lem}

\begin{proof} \ (1)\ Since $f$ itself is an
$\infty$-extensible locally bijective pseudo-extensions of $h$, by
Definition \ref{degree-ext}, we have $\mathrm{deg}(f)=\mathrm{deg}(h)$. By Lemma
\ref{inverse}, we have $\mathrm{deg}(h) \in \{1, -1\}$.

\vspace{3mm} (2)\ By the conclusion (1) of this lemma, we get
$\mathrm{deg}(g)=\mathrm{deg}(f)=\mathrm{deg}(h)$.
\end{proof}

 Clearly, there exist a continuous map $f: \mathbb{R}^n
\rightarrow \mathbb{R}^n$ and two $n$-balls $V$ and $W$ in
$\mathbb{R}^n$ such that $f$ is bijective on $V \cup W$, but
$\mathrm{deg}(f|W)=-\mathrm{deg}(f|V)$. Hence, in Lemma \ref{uniform}, the
condition that $f: \mathbb{R}^n \rightarrow \mathbb{R}^n$ is
$\infty$-extensible cannot be removed.

\begin{lem} \label{equivalence} Let $Y$ be an $n$-ball
in $\Bbb R^n$, $n \geq 1$, $V$ be an $n$-ball in
$\stackrel{\circ}Y$, and $f: (Y, \partial Y) \rightarrow (Y,
\partial Y)$ be a continuous map. Suppose that $f$ is bijective on
$V$. Put $h=f|V: V \rightarrow Y$. Let $\mathrm{deg}(h)$ be defined
as in Definition \ref{degree-ext}, and let $\mathrm{deg}(f)$ be defined as in
(2) of Definition \ref{class-deg}. Then $\mathrm{deg}(f)=\mathrm{deg}(h)$.
\end{lem}

\begin{proof}  Take a homeomorphism $\eta: Y
\rightarrow B^n$. Write $V'=\eta(V)$. Let $f'=\eta \circ f \circ
\eta^{-1}: B^n \rightarrow B^n$, and let $h'=\eta \circ h \circ
\eta^{-1}|V': V' \rightarrow B^n$. Then by Lemma \ref{conj-deg} and Lemma \ref{deg-conj}
we have $\mathrm{deg}(f')=\mathrm{deg}(f)$ and
$\mathrm{deg}(h')=\mathrm{deg}(h)$. Take $0<r<1$ such that $V' \cup
f'(V') \subset B^n(0,r)$. Let $\mu: [0,1) \rightarrow [0,\infty)$ be
a homeomorphism such that $\mu|[0, r]=id$. Define a continuous map
$g: (B^n, \partial B^n) \rightarrow (\Bbb R_{\;\infty}^{\;n},
\{\infty\})$ by $g(\partial B^n)=\{\infty\}$, and $g(t \cdot
x)=\mu(t) \cdot x$ for $x \in \partial B^n$ and $t \in
[0,1)$. Let $\varphi(x)=g \circ f' \circ g^{-1}(x)$ if $x\in\mathbb R^n$ and $\varphi(\infty)=\infty$.
Then $\varphi\circ g=g\circ f'$.  Noting that $g_*: H_n(B^n,
\partial B^n) \rightarrow H_n(\Bbb
R_{\;\infty}^{\;n}, \{\infty\})$ is an isomorphism, we have
$\mathrm{deg}(\varphi)=\mathrm{deg}(f')$. Let $\varphi_1=\varphi:
\Bbb R_{\;\infty}^{\;n} \rightarrow \Bbb R_{\;\infty}^{\;n}$, and
$\varphi_0=\varphi|\Bbb R^n: \Bbb R^n \rightarrow \Bbb R^n$. Then
$\mathrm{deg}(\varphi_0)=\mathrm{deg}(\varphi_1)=\mathrm{deg}(\varphi)$.
Since $\varphi_0|V'=h'$, $\varphi_0$ is an $\infty$-extensible
locally bijective pseudo-extensions of $h'$. Thus
$\mathrm{deg}(h')=\mathrm{deg}(\varphi_0)$, and hence
$\mathrm{deg}(f)=\mathrm{deg}(h)$.
\end{proof}

In Definition \ref{degree-ext}, $V$ is an $n$-ball in
$\mathbb{R}^n$ and $h$ is a continuous injection from $V$ to a
subspace $W$ of $\mathbb{R}^n$. In the following definition, we will
consider $n$-balls in general topological spaces.

\begin{defin}\label{degree-3} (1)\ Let $Y$ and $V$ be
two $n$-balls in some topological space $X$ with $V \subset Y$, $n
\geq 1$, and $h: V \rightarrow Y$ be a continuous injection. Take an
imbedding $\eta: Y \rightarrow \mathbb{R}^n$. Let $Y'=\eta(Y)$,
$V'=\eta(V)$, and let $h'=\eta \circ h \circ \eta^{-1}: V'
\rightarrow Y'$. Then we define the {\it \textbf{degree}} of $h$ to
be $\mathrm{deg}(h)=\mathrm{deg}(h')$. It follows from Lemma \ref{deg-conj}
that the definition of $\mathrm{deg}(h)$ is independent of the
choice of $\eta$.

\vspace{3mm} (2)\ Let $S$ be an $n$-sphere, $n \geq 1$, $V$ be an
$n$-ball in $S$, and $h:V \rightarrow S$ be a continuous injection.
Take an $n$-ball $D \subset V$ such that $S-D-h(D) \neq \emptyset$,
and take an $n$-ball $Y \subset S$ such that $D \cup h(D) \subset
Y$. Let $h_D=h|D: D \rightarrow Y$, and then we define the {\it
\textbf{degree}} of $h$ to be $\mathrm{deg}(h)=\mathrm{deg}(h_D)$.
It is easy to see from Lemma \ref{restrict} and (1) of Remark \ref{rem1} that the
definition of $\mathrm{deg}(h)$ is independent of the choice of $D$
and $Y$.
\end{defin}

\begin{rem} Let $Y$, $Z$, $W$ and $V$ be
$n$-balls in some topological space $X$ with $V \subset Y \cap Z$,
$n \geq 1$. Let $h: V \rightarrow Y$ and $h':V \rightarrow Z$ be
continuous injections. Suppose that $h(V)=h'(V)=W$ and $h(x)=h'(x) \in Y \cap Z$ for
any $x \in V$. Then

\vspace{3mm} $(1)$\ It is possible that
$\mathrm{deg}(h')=-\mathrm{deg}(h)$. For example, if $X$ is a Mobius
strip, then there exist $2$-balls $Y$, $Z$, $V$, $W$ and continuous
injections $h: V \rightarrow Y$ and $h': V \rightarrow Z$ such that
$Y \cup Z=X$, $Y \cap Z \supset V \cup W$, $h(V)=h'(V)=W$,
$h(x)=h'(x)$ for any $x \in V$, and
$\mathrm{deg}(h')=-\mathrm{deg}(h)$. Hence, in Definition \ref{degree-3}, in
general, the definition of the degree of the continuous injection
$h: V \rightarrow Y$ is dependent on the choice of $Y$.

\vspace{3mm} $(2)$\ On the other hand, if there exists an $n$-ball
$Q \subset X$ such that $V \cup W \subset Q \subset Y \cap Z$ or $Q
\supset Y \cup Z$, then by Lemma \ref{restrict} it is easy to show
$\mathrm{deg}(h')=\mathrm{deg}(h)$. Hence, if $V \cup W \subset
\mathbb{R}^n$ and $Y$ is restricted to be in $\mathbb{R}^n$, then
the definition of the degree of the continuous injection $h:V
\rightarrow Y$ is independent of the choice of $Y$.

\vspace{3mm} $(3)$\ Let $M$ be a connected orientable $n$-manifold,
$n \geq 1$, $V$ be a connected $n$-submanifold of $M$, and $h: V
\rightarrow M$ be a continuous injection. Similar to $(2)$ of
Definition \ref{degree-3}, we can also take an $n$-ball $D \subset V$ and an
$n$-ball $Y \subset M$ such that $D \cup h(D) \subset Y$, and then
we define the {\it \textbf{degree}} of $h$ to be
$\mathrm{deg}(h)=\mathrm{deg}(h_D)$. The detail is omitted here.
\end{rem}

\section{Blockading sets and fixed points}

A subset $A$ of a space $X$ is called a {\it
\textbf{retract}} of $X$ if there exists a continuous map $r: X
\rightarrow A$ such that $r(x)=x$ for each $x \in A$, and such a map
$r: X \rightarrow A$ is called a {\it \textbf{retraction}} of $X$ to
$A$. It is easy to show that if $A$ is a retract of a normal space
$X$ then $A$ must be a closed subset of $X$. A normal space $Y$ is
called an {\it \textbf{absolute retract}} if for any normal space
$Z$ and any closed subset $W$ of $Z$, whenever $W$ is homeomorphic
to $Y$ then $W$ is a retract of $Z$ \cite[p. 221]{Mun75}. It is
well known that every ball is an absolute retract. Note that not
all absolute retracts are compact. For example, every open ball is
also an absolute retract. In addition, by the definition, the empty
set $\emptyset$ cannot be an absolute retract.
\medskip

It is well known that each $(n-1)$-sphere $S$ in $\mathbb R^n$ separates
$\mathbb R^n$ into exactly two connected components: one is bounded and denoted by
${\rm Int(S)}$, the other is unbounded and denoted by ${\rm Ext(S)}$ (see e.g. \cite{HY61}).
\medskip

 The following retraction theorem is due to Bing \cite[Theorem 2]{Bi64}.

\begin{thm}\label{bing}
For each $(n-1)$-sphere $S$ in $\mathbb R^n$ and each point $p\in {\rm Int}(S)$,
there is a retraction of $\mathbb R^n-\{p\}$ to $S$.
\end{thm}

\begin{defin}\label{block}Let $X$ be an $n$-ball in
$\Bbb R^n$, $n \geq 1$, and $f: X \rightarrow \mathbb{R}^n$ be a
continuous map. A subset $V$ of $X$ is called an {\it
\textbf{$(f,X)$-blockading set}} if there exists an open
neighborhood $U$ of $f^{-1}(V)$ in $X$ such that $f(U) \subset X$.
\end{defin}

 Clearly, in Definition \ref{block}, if $V=\emptyset$ or $f(X)
\cap V=\emptyset$, then $V$ is an $(f, X)$-blockading set.

\begin{thm} \label {main1'} Let $X$ be an $n$-ball
in $\Bbb R^n$, $n \geq 1$, $Y\subset \partial X$ be an absolute
retract, and $f: X \rightarrow \mathbb{R}^n$ be a continuous map. If
$f(Y) \subset X$, and $\partial X-Y$ is an $(f, X)$-blockading set,
then $f$ has a fixed point in $X$.
\end{thm}

\begin{proof}
Noting that $Y \subset \partial X$ and $\partial X$ is closed in $\Bbb R^n$,  $Y$ is a closed subset of $(\Bbb R^n-X)\cup Y$, no
matter whether $Y$ is a closed subset of $\partial X$ or not. Since $Y$ is an absolute retract, there is a
retraction $\alpha: (\Bbb R^n-X)\cup Y \rightarrow Y$. Define a map $\beta: \Bbb R^n
\rightarrow X$ by $\beta|(\Bbb R^n-X)\cup Y=\alpha$ and
$\beta|X=id$. Then $\beta$ is not continuous at any point $v \in
\partial X-\overline{Y}$, and $\beta$ may not be continuous at any point $v \in \overline{Y}-Y$,
but $\beta$ is continuous at any point $v \in \Bbb R^n-(\partial
X-Y)$. Let $h=\beta \circ f: X \rightarrow X$. We have
\medskip

\textbf{Claim 3.3.1.}\ $h:X \rightarrow X$ is a continuous map.
\medskip

{\textbf{Proof of Claim 3.3.1.}} \ Consider any point $x \in X$. If
$f(x) \in \Bbb R^n-(\partial X-Y)$, then, since $f$ is continuous at
$x$ and $\beta$ is continuous at $f(x)$, it follows that $h=\beta
\circ f$ is continuous at $x$. If $f(x) \in
\partial X-Y$, then, since $\partial X-Y$ is $(f, X)$-blockading,
there is an open neighborhood $U_x$ of $x$ in $X$ such that $f(U_x)
\subset X$, which with $\beta|X=id$ implies that $h|U_x=\beta \circ
f|U_x=f|U_x$ is continuous. Thus $h$ is continuous at any point $x
\in X$.

\medskip \textbf{Claim 3.3.2.}\
$\mathrm{Fix}(f)=\mathrm{Fix}(h)$.
\medskip

{\textbf{Proof of Claim 3.3.2.}} \ If $x \in \mathrm{Fix}(f)$, then
$f(x)=x \in X$, which with $\beta|X=id$ implies $h(x)=\beta \circ f
(x)=f(x)=x$. Thus $\mathrm{Fix}(f) \subset \mathrm{Fix}(h)$.

\vspace{3mm} Conversely, consider any $y \in \mathrm{Fix}(h)$. If
$f(y) \in \Bbb R^n-X$, then $y=h(y)=\beta \circ f(y) \in \beta(\Bbb
R^n-X)=\alpha(\Bbb R^n-X) \subset Y$. However, from the condition of the
theorem, we get $f(y) \in f(Y)\subset X$, which is a contradiction.
 Thus we must have $f(y) \in X$, which with $\beta|X=id$ implies that $f(y)=\beta
\circ f(y)=h(y)=y$. Hence $\mathrm{Fix}(h) \subset \mathrm{Fix}(f)$
also holds. Claim 3.3.2 is proved.

\vspace{3mm} From Claim 3.3.2, Claim 3.3.1, and Theorem
\ref{Brouwer}, we obtain $\mathrm{Fix}(f)=\mathrm{Fix}(h) \neq
\emptyset$.
\end{proof}

The following corollary is immediate.

\begin{cor} Let $X$ be an
$n$-ball in $\Bbb R^n$, $n \geq 1$, $Y\subset \partial X$ be an
absolute retract, and $f: X \rightarrow \mathbb{R}^n$ be a
continuous map. If $f(Y) \subset X$, and $f(X) \cap (\partial
X-Y)=\emptyset$, then $f$ has a fixed point.
\end{cor}

\vspace{3mm} Considering the case that $Y$ is an $(n-1)$-ball in
$\partial X$, we have the following theorem which will be used in the next section.

\begin{thm}\label{ball} Let $X$ be an $n$-ball
in $\Bbb R^n$, $n \geq 1$, $Y$ be an $(n-1)$-ball in $\partial X$,
and $f: X \rightarrow \mathbb{R}^n$ be a continuous map. If
$f^{-1}(\partial X-Y)=\stackrel{\circ}Y$, and there exist a point $v
\in \stackrel{\circ}Y$ and an open neighborhood $W_v$ of $v$ in $X$
such that $f(W_v) \subset X$, then $f$ has a fixed point.
\end{thm}

\begin{proof} Note that the condition
$f^{-1}(\partial X-Y)=\stackrel{\circ}Y$ in the theorem is
equivalent to $f(\stackrel{\circ}Y) \subset \partial X-Y$ and
$f(X-\stackrel{\circ}Y) \cap (\partial X-Y)=\emptyset$. Thus, for
each $y \in \stackrel{\circ}Y$, there exists a connected open
neighborhood $U_y$ of $y$ in $X$ such that $U_y \cap
\partial X \subset \stackrel{\circ}Y$ and
$f(U_y-\stackrel{\circ}Y)\cap \partial X=\emptyset$. Write
$U=\bigcup \{U_y: y \in \stackrel{\circ}Y\}$. Then $U$ is a
connected open neighborhood of $\stackrel{\circ}Y$ in $X$, $U \cap
\partial X=\stackrel{\circ}Y$, $f(U-\stackrel{\circ}Y)\cap \partial X
=\emptyset$, and $U-\stackrel{\circ}Y$ is also connected. Since the
connected set $U-\stackrel{\circ}Y$ has a nonempty subset
$(U-\stackrel{\circ}Y) \cap W_v$ satisfying
$f\big((U-\stackrel{\circ}Y)\cap W_v\big) \subset
\stackrel{\circ}X$, we have $f(U- \stackrel{\circ}Y)\subset
\stackrel{\circ}X$, and hence $f(U) \subset X$. Therefore, by
Definition \ref{block}, $\partial X-Y$ is an $(f, X)$-blockading set, and by
Theorem \ref{main1'}, $f$ has a fixed point.
\end{proof}

\vspace{3mm} From Theorem \ref{ball}, we get

\begin{cor} Let $X$ be an
$n$-ball in $\Bbb R^n$, $n \geq 2$, $Y$ be an $(n-1)$-ball in
$\partial X$, and $f: X \rightarrow \mathbb{R}^n$ be a continuous
injection. If $f(Y)=\partial X-\stackrel{\circ}Y$, and there exist a
point $v \in \stackrel{\circ}Y$ and an open neighborhood $W_v$ of
$v$ in $X$ such that $f(W_v) \subset X$, then $f$ has a fixed
point.
\end{cor}

\begin{proof} Since $f: X \rightarrow \Bbb R^n$
is a continuous injection, it follows from $f(Y)=\partial
X-\stackrel{\circ}Y$ that $f^{-1}(\partial X-Y)=\stackrel{\circ}Y$.
\end{proof}

\section{Degrees of maps and fixed points}

In this section, we will use several notions of degree of a map given in Section 2 and Theorem \ref{ball} to establish some
fixed point theorems for continuous maps from an $n$-ball to $\mathbb R^n$.
\medskip

The following lemma is obviously true. However, we
still give a simple proof of this lemma, as an example computing the
degrees of some special maps.

\begin{lem} \label{reflect} For given $n \geq 1$,
define the reflection $\gamma_n: \mathbb{R}^n \rightarrow
\mathbb{R}^n$ by $\gamma_n(y, r_n)=(y, -r_n)$ for any $y \in
\mathbb{R}^{n-1}$ and any $r_n \in \mathbb{R}$. Then
$\mathrm{deg}(\gamma_n)=-1$.
\end{lem}

\begin{proof} Take a simplicial complex $K$ in
$\mathbb{R}^{n+1}$ such that if $\bigtriangleup$ is a simplex of $K$
then $\gamma_{n+1}(\bigtriangleup)$ is also a simplex of $K$, and
the polyhedron $|K|$ of $K$ is an $n$-sphere. Take a homeomorphism
$h: \Bbb R_{\;\infty}^{\;n} \rightarrow |K|$ such that $\gamma_{n+1}
\circ h(x)=h \circ \gamma_n(x)$ for any $x \in \mathbb{R}^n$. Define
a map $\beta: |K| \rightarrow |K|$ by $\beta \circ
h(\infty)=h(\infty)$ and $\beta(y)=h \circ \gamma_n \circ h^{-1}(y)$
for any $y \in h(\mathbb{R}^n)$. Then $\beta$ is a homeomorphism,
and both $\beta$ and $\beta^{-1}$ are simplicial maps for the
triangulation $K$ of $|K|$. Let $\beta_*: H_n(K) \rightarrow H_n(K)$
be the isomorphism induced by $\beta$. By computing we can obtain
$\beta_*(\alpha)=-\alpha$ for any $\alpha \in H_n(K)$. Thus
$\mathrm{deg}(\gamma_n)=\mathrm{deg}(\beta)=-1$.
\end{proof}

Let $X$ be a topological space. The {\it\textbf{suspension}} $\Sigma X$ of $X$ is the quotient space of $X\times [-1, 1]$ under
an identification of $X\times \{1\}$ to a point $c_+$ and $X\times\{-1\}$ to a different point $c_-$.
Any $f:X\rightarrow Y$  induces a map $\Sigma f:\Sigma X\rightarrow \Sigma Y$ defined by $\Sigma f(x, t)=(f(x), t)$
for $x\in X$ and $t\in [-1, 1]$.

\medskip
Recall that if $f:X\rightarrow Y$ is continuous, then for each $i$, $f$ induces a homomorphism
$f_i:{\tilde H}_i(X)\rightarrow {\tilde H}_i(Y)$. The following theorem can be seen in \cite[p.29-Theorem 3.2.13]{We14}.

\begin{thm} \label {suspension} (1) For any space $X$ there is an isomorphism $\Sigma: {\tilde H}_{i+1}(\Sigma X)\rightarrow {\tilde H}_i(X)$
for any $i$.

(2) For any $f:X\rightarrow Y$ and $i$, $f_i\circ \Sigma=\Sigma\circ (\Sigma f)_{i+1}$.
\end{thm}

\begin{lem}\label{equal2} Let $\varphi:
\mathbb{R}^{n-1} \rightarrow \mathbb{R}^{n-1}$ be an
$\infty$-extensible continuous map, $n \geq 2$. Define $f:
\mathbb{R}^n \rightarrow \mathbb{R}^n$ by $f(x, r_n)=(\varphi(x),
r_n)$ for any $x \in \mathbb{R}^{n-1}$ and any $r_n \in \mathbb{R}$.
Then $\mathrm{deg}(f)=\mathrm{deg}(\varphi)$.
\end{lem}

\begin{proof}
Let $\tilde \varphi:\mathbb R^{n-1}_\infty\rightarrow \mathbb R^{n-1}_\infty$
and $\tilde f:\mathbb R^{n}_\infty\rightarrow \mathbb R^{n}_\infty$ be the natural extensions of $\varphi$
and $f$ respectively. Let $K=\Sigma \mathbb R^{n-1}_\infty$ be the suspension of $\mathbb R^{n-1}_\infty$.
By Theorem \ref{suspension}, there is an isomorphism  $\Sigma: {\tilde H}_{n}(K)\rightarrow {\tilde H}_{n-1}(\mathbb R^{n-1}_\infty)$
such that $\tilde\varphi_{n-1}\circ \Sigma=\Sigma\circ (\Sigma \tilde\varphi)_{n}$. This implies ${\rm deg}(\tilde\varphi)={\rm deg}(\Sigma \tilde\varphi)$.

Let $J=\{\infty\}\times [-1, 1]\subset K$ and let $K/J$ be the quotient space by identifying $J$ to a point
and let $\psi:K\rightarrow K/J$ be the quotient map. Since $J$ is contractible, $\psi_n: \tilde H_n(K)\rightarrow \tilde H_n(K/J)$
is an isomorphism by the exactness of the homology sequence and the isomorphism $\tilde H_n(K/J)\cong H_n(K, J)$ (see \cite[p.28-Theorem 3.2.9 (2)]{We14}). By the definition of $\Sigma \tilde\varphi$, there is a map $\xi:K/J\rightarrow K/J$ with
$\psi\circ\Sigma\tilde\varphi=\xi\circ \psi$. Then we have ${\rm deg}(\Sigma \tilde\varphi)={\rm deg}(\xi)$.

Take an increasing homeomorphism $\mu: (-1, 1) \rightarrow \mathbb{R}$.
Define a map $\eta: K
\rightarrow \mathbb R_{\;\infty}^{\;n}$ by
$\eta(J)=\{\infty\}$ and $\eta(y, t)=\big(y, \mu(t)\big)$ for any
$y \in K-J$ and any $t \in  (-1, 1)$. Then $\eta$ induces a
homeomorphism $\tilde\eta:K/J\rightarrow \mathbb R_{\;\infty}^{\;n}$.
It is direct to check that $\tilde f\circ \tilde \eta=\tilde\eta\circ\xi$.
So ${\rm deg}(\xi)={\rm deg}(\tilde f)$.

From the above arguments, we see that ${\rm deg}(\tilde\varphi)={\rm deg}(\tilde f)$, which means
${\rm deg}(\varphi)={\rm deg(f)}$.
\end{proof}

Define projections $p: \Bbb R^n \rightarrow \Bbb
R^{n-1}$ and $q: \Bbb R^n \rightarrow \Bbb R$ by
$$\hspace{48mm} p(y,r)=y, \ \ \ \ \ \ q(y,r)=r, \hspace{56mm} \mbox{(3.1)}$$
for any $y \in \Bbb R^{n-1}$ and any $r \in \Bbb R$. Write $\Bbb
R_+^{\,n}=q^{-1}\big([0, \infty)\big)$, $\Bbb
R_-^{\,n}=q^{-1}\big((-\infty, 0]\big)$. We have

\begin{lem}\label{ball-degree} Let $E$ be an
$(n-1)$-ball in $\mathbb{R}^{n-1}$ and $D$ be an $n$-ball in
$\mathbb{R}_+^{\,n}$ such that $D \cap \mathbb{R}^{n-1}=E$, $n \geq
2$. Let $f: D \rightarrow \mathbb{R}^n$ be a continuous injection
which satisfies $f(D) \cap \mathbb{R}^{n-1}=f(E)$. Write $f_E=f|E: E
\rightarrow \mathbb{R}^{n-1}$. Then

\vspace{3mm} $(1)$\ $f(D) \subset \mathbb{R}_+^{\,n}$ or $f(D)
\subset \mathbb{R}_-^{\,n}$;

\vspace{3mm} $(2)$\ If $f(D) \subset \mathbb{R}_+^{\,n}$ then
$\mathrm{deg}(f)=\mathrm{deg}(f_E)$, and if $f(D) \subset
\mathbb{R}_-^{\,n}$ then $\mathrm{deg}(f)=-\mathrm{deg}(f_E)$.
\end{lem}

\begin{proof} (1) is trivial.

\vspace{3mm} (2) Let
$$\lambda=\left\{\begin{array}{cl} 1, &
\mbox{\ if\ } f(D) \subset \mathbb{R}_+^{\,n};\\
-1, & \mbox{\ if\ } f(D) \subset
\mathbb{R}_-^{\,n}.\end{array}\right.$$ Write $P=E \times [-1,0]$.
Define $f_P: P \rightarrow \mathbb{R}^n$ by $f_P(x,
r_n)=\big(f_E(x), \lambda r_n\big)$ for any $x \in E \subset
\mathbb{R}^{n-1}$ and any $r_n \in [-1, 0]$. Write $Q=D \cup P$.
Define $f_Q: Q \rightarrow \mathbb{R}^n$ by $f_Q|D=f$ and
$f_Q|P=f_P$. Noting that $f_Q$ is a continuous injection, by Lemma
\ref{restrict} we have $\mathrm{deg}(f)=\mathrm{deg}(f_Q)=\mathrm{deg}(f_P)$.

\vspace{3mm} Let $\varphi: \mathbb{R}^{n-1} \rightarrow
\mathbb{R}^{n-1}$ be an $\infty$-extendible locally bijective
pseudo-extension of $f_E$. Define $\phi=\varphi \times
id_{\mathbb{R}}: \mathbb{R}^n \rightarrow \mathbb{R}^n$ by $\phi(x,
r_n)=(\varphi(x), r_n)$ for any $x \in \mathbb{R}^{n-1}$ and any
$r_n \in \mathbb{R}$. Then by Lemma \ref{equal2} we have
$\mathrm{deg}(\phi)=\mathrm{deg}(\varphi)=\mathrm{deg}(f_E)$.

\vspace{3mm} Let the reflection $\gamma_n: \mathbb{R}^n \rightarrow
\mathbb{R}^n$ be the same as in Lemma \ref{reflect}. Define $\psi:
\mathbb{R}^n \rightarrow \mathbb{R}^n$ by
$$\psi=\left\{\begin{array}{cl} \phi, &
\mbox{\ if\ } f(D) \subset \mathbb{R}_+^{\,n};\\
\gamma_n \circ \phi, & \mbox{\ if\ } f(D) \subset
\mathbb{R}_-^{\,n}.\end{array}\right.$$ Then
$\mathrm{deg}(\psi)=\lambda \cdot \mathrm{deg}(\phi)$, and $\psi$ is
an $\infty$-extensible locally bijective pseudo-extension of $f_P$,
and hence $\mathrm{deg}(\psi)=\mathrm{deg}(f_P)$. To sum up, we get
$\mathrm{deg}(f)=\mathrm{deg}(f_D)=\mathrm{deg}(f_P)=\mathrm{deg}(\psi)=\lambda
\cdot \mathrm{deg}(\phi)=\lambda \cdot \mathrm{deg}(f_E)$.
\end{proof}

\begin{prop} \label{degree-img} Let $X$ be an
$n$-ball in $\mathbb{R}^n$, $n \geq 2$, and $f: X \rightarrow
\mathbb{R}^n$ be a continuous map. Suppose that there exist an
$n$-ball $D$ in $X$ and an $(n-1)$-ball $E$ in $\partial X$ such
that $D \cap \partial X=E$, $f(D) \cap
\partial X=f(E)$, and $f$ is bijective on $D$. Let $f_D=f|D: D
\rightarrow \mathbb{R}^n$ and $f_E=f|E: E \rightarrow \partial X$.
Then

\vspace{3mm} $(1)$\ $f(D-E) \subset \stackrel{\circ}X$ or $f(D-E)
\subset \Bbb R^n-X$;

\vspace{3mm} $(2)$\ If $f(D-E) \subset \stackrel{\circ}X$ then
$\mathrm{deg}(f_D)=\mathrm{deg}(f_E)$, and if $f(D-E) \subset \Bbb
R^n-X$ then $\mathrm{deg}(f_D)=-\mathrm{deg}(f_E)$. $($See (2) of Definition \ref{degree-3}
for the meaning of $\mathrm{deg}(f_E)$.$)$
\end{prop}

\begin{proof} (1) It is clear, since $f(D) \cap
\partial X=f(E) \subset \partial X$, and $f$ is bijective on $D$.

\vspace{3mm} (2)\ We may assume that $E$ is benign in $\mathbb \partial D$ (that is $\partial D-\stackrel{\circ}E$
is also an $(n-1)$-ball; see Remark \ref{rem1}), and $\partial X-E-f(E) \neq \emptyset$,
since, if not, we can take an $n$-ball $D_0 \subset D$ and
$(n-1)$-ball $E_0 \subset E$ such that $D_0 \cap \partial X=E_0$,
$\partial X-E_0-f(E_0) \neq \emptyset$, and $\partial
D_0-\stackrel{\circ}E_0$ is an $(n-1)$-ball, and then we can replace
$D$ by $D_0$.

\vspace{3mm} Take an $(n-1)$-ball $Q$ in $\partial X$ such that $E
\cup f(E) \subset Q$ and $\partial X-\stackrel{\circ}Q$ is also an
$(n-1)$-ball, and then take an imbedding $\eta: X \rightarrow
\mathbb{R}_+^{\,n}$ such that $\eta(X) \cap
\mathbb{R}^{n-1}=\eta(Q)$. Write $X'=\eta(X)$, $D'=\eta(D)$, and
$E'=\eta(E)$. Let $f_{E'}=\eta \circ f \eta^{-1}|E': E' \rightarrow
\partial X'$. Then $f_{E'}(E')=\eta \circ f(E) \subset \eta(Q) \subset
\mathbb{R}^{n-1}$.

\vspace{3mm} (a)\ If $f(D-E) \subset \stackrel{\circ}X$, then, put
$f_{D'}=\eta \circ f \circ \eta^{-1}|D': D' \rightarrow X'$. From
Lemmas \ref{deg-conj} and \ref{ball-degree}, we get
$\mathrm{deg}(f_D)=\mathrm{deg}(f_{D'})=\mathrm{deg}(f_{E'})=\mathrm{deg}(f_E)$.

\vspace{3mm} (b)\ If $f(D-E) \subset \Bbb R^n-X$, then, write
$W=f(D)$, $W'=f_{E'}(E') \times [-1, 0]$, and take a homeomorphism
$\eta_0: W \rightarrow W'$ such that $\eta_0|f(E)=\eta|f(E)$. Let
$f_{D'}=\eta_0 \circ f \eta^{-1}|D': D' \rightarrow W'$ (the existence of such $\eta_0$ is implied by
the benignity of $E$ in $\partial D$ and the injectivity of $f|D$; see Remark \ref{rem1}). Then by
Lemmas \ref{deg-conj} and \ref{ball-degree}, and Definition \ref{degree-3}, we get
$\mathrm{deg}(f_D)=\mathrm{deg}(f_{D'})=-\mathrm{deg}(f_{E'})=-\mathrm{deg}(f_E)$.
\end{proof}

\begin{prop}\label{deg-int} Let $X$ be an
$n$-ball in $\mathbb{R}^n$, $n \geq 2$, $Y$ be an $(n-1)$-ball in
$\partial X$, and $D$ be an $n$-ball in $X$ such that $E\equiv
\partial D \cap \partial X$ is an $(n-1)$-ball in
$\stackrel{\circ}Y$. Let $f: X \rightarrow \mathbb{R}^n$ be a
continuous map such that $f(Y) \subset \partial
X-\stackrel{\circ}Y$, $f(\partial Y) \subset \partial Y$, $f(D) \cap
\partial X=f(E) \subset \partial X-Y$, and $f$ is bijective on $D$.
Let $f_D=f|D: D \rightarrow \mathbb{R}^n$ and $f_{\partial
Y}=f|\partial Y: \partial Y \rightarrow
\partial Y$.  Then

\vspace{3mm} $(1)$\ $f(D-E) \subset \stackrel{\circ}X$ or $f(D-E)
\subset \Bbb R^n-X$;

\vspace{3mm} $(2)$\ If $f(D-E) \subset \stackrel{\circ}X$ then
$\mathrm{deg}(f_D)=-\mathrm{deg}(f_{\partial Y})$. If $f(D-E)
\subset \Bbb R^n-X$ then $\mathrm{deg}(f_D)=\mathrm{deg}(f_{\partial
Y})$.
\end{prop}

\begin{proof} (1)\ The proof is similar to
Proposition \ref{degree-img}.

\vspace{3mm} (2)\ Let $f_E=f|E: E \rightarrow \partial X$. Then the
conclusion (2) of Proposition \ref{degree-img} is still true. Further, we have

\vspace{3mm} {\textbf{Claim 1.}}\
$\mathrm{deg}(f_E)=-\mathrm{deg}(f_{\partial Y})$.

\vspace{3mm} {\textbf{Proof of Claim 1.}}\ Write $W=f(E)$. Then
$W$ is an $(n-1)$-ball in $\partial X-Y$. Take an $(n-1)$-ball $V
\subset \stackrel{\circ}W$ such that $\partial X-\stackrel{\circ}V$
is also an $(n-1)$-ball. Write $Z=\partial X-\stackrel{\circ}V$.
Then $\stackrel{\circ}Z \supset Y$, $\partial Z=\partial V$, and
there exists a homeomorphism $\eta:
\partial X \rightarrow S^{n-1}$ such that $\eta(\partial
V)=S^{n-2}$ (see (2) of Remark \ref{rem1}). Let the reflection $\gamma_n: \mathbb{R}^n \rightarrow
\mathbb{R}^n$ be as in Lemma \ref{reflect}. Let $h'=\gamma_n|S^{n-1}: S^{n-1}
\rightarrow S^{n-1}$, and let $h=\eta^{-1} \circ h' \circ \eta:
\partial X \rightarrow \partial X$. Then $h$ is a homeomorphism,
$h|\partial V=id$, $h(V)=Z$, $h(Z)=V$, and $h^2=id$. Write
$V_0=h^{-1}(E)$. Then $V_0 \subset h^{-1}(\stackrel{\circ}Y)\subset
h^{-1}(\stackrel{\circ}Z)=\stackrel{\circ}V$. Let $h_0=h|V_0: V_0
\rightarrow \partial X$. Then
$\mathrm{deg}(h_0)=\mathrm{deg}(h)=\mathrm{deg}(h')=-1$. Let
$E_0=f^{-1}(V_0)$. Since $f$ is bijective on $D$ and $E_0 \subset E
\subset \partial D$, it follows that $E_0 \subset f^{-1}(V) \subset
f^{-1}(\stackrel{\circ}W)=\stackrel{\circ}E$, and $E_0$ is an
$(n-1)$-ball. Let $f_0=f|E_0: E_0 \rightarrow \partial X$. Then
$f_0$ is an injection, $f_0(E_0)=V_0$, and
$\mathrm{deg}(f_0)=\mathrm{deg}(f_E)$. Let $g=h_0 \circ f_0: E_0
\rightarrow \partial X$. Then $g(E_0)=h_0(V_0)=E$, and \vspace{-1mm}
$$\mathrm{deg}(g)=\mathrm{deg}(h_0) \cdot \mathrm{deg}(f_0)=
\mathrm{deg}(h_0) \cdot \mathrm{deg}(f_E)=-\mathrm{deg}(f_E).$$

It follows from Theorem \ref{bing} that  there is a retraction $\alpha: Z \rightarrow Y$ such that
$\alpha(Z-Y) \subset \partial Y$. Define $\psi: \partial
X-\stackrel{\circ}Y \rightarrow Z$ by $\psi|V=h|V$ and
$\psi|(Z-\stackrel{\circ}Y)=id$. Let $f_Y=f|Y: Y \rightarrow
\partial X-\stackrel{\circ}Y$. Let $\xi=\alpha \circ \psi \circ f_Y: Y \rightarrow
Y$. Since $\psi \circ f_Y(E_0)=\psi \circ
f_0(E_0)=\psi(V_0)=h(V_0)=E$, we have $\xi(E_0)=\alpha \circ \psi
\circ f_Y(E_0)=\alpha(E)=E \subset \stackrel{\circ}Y$, and
$\xi|E_0=\psi \circ f|E_0=h \circ f|E_0=g: E_0 \rightarrow Y \subset
X$. Since $f_Y$ is bijective on $E_0$, $\psi$ is bijective on
$V_0=f_Y(E_0)$, and $\alpha$ is bijective on $\psi \circ
f_Y(E_0)=\psi(V_0)=E$, it follows that $\xi$ is bijective on $E_0$,
and $\mathrm{deg}(\xi|E_0)=\mathrm{deg}(g)$.

\vspace{3mm} Since $f_Y(\partial Y)=f(\partial Y) \subset \partial Y
\subset \stackrel{\circ}Z-\stackrel{\circ}Y$, we have $\psi \circ
f_Y|\partial Y=f|\partial Y$ and $\psi \circ f_Y(\partial
Y)=f(\partial Y) \subset \partial Y$, which with $\alpha|\partial
Y=id$ imply $\xi(\partial Y)=\alpha \circ \psi \circ f_Y(\partial Y)
\subset \partial Y$ and $\xi|\partial Y=\alpha \circ \psi \circ
f_Y|\partial Y=\psi \circ f_Y|\partial Y=f_Y|\partial Y=f|_{\partial
Y}: \partial Y \rightarrow \partial Y$. Hence, by Lemmas \ref{deg-bound} and
\ref{equivalence} we have \vspace{-1mm}
$$\mathrm{deg}(f_{\partial Y})=\mathrm{deg}(\xi|\partial Y)=\mathrm{deg}(\xi)=\mathrm{deg}(\xi|E_0).$$

To sum up, we obtain $\mathrm{deg}(f_{\partial
Y})=\mathrm{deg}(\xi|E_0)=\mathrm{deg}(g)=-\mathrm{deg}(f_E)$. Thus the Claim
1 is proved.

\vspace{3mm} Note that the Claim 1 with the conclusion (2) of
Proposition \ref{degree-img} implies the conclusion (2) of Proposition \ref{deg-int}.
\end{proof}

\begin{thm} \label{main2'}Let $X$ be an $n$-ball
in $\mathbb{R}^n$, $n \geq 2$, $Y$ be an $(n-1)$-ball in $\partial
X$, and $D$ be an $n$-ball in $X$ such that $E\equiv
\partial D \cap \partial X$ is an $(n-1)$-ball in
$\stackrel{\circ}Y$. Let $f: X \rightarrow \mathbb{R}^n$ be a
continuous map such that $f$ is bijective on $D$, and
$f^{-1}(\partial X-Y)=\stackrel{\circ}Y$. Let $f_D=f|D: D
\rightarrow \mathbb{R}^n$ and $f_{\partial Y}=f|\partial Y: \partial
Y \rightarrow \partial Y$. If
$\mathrm{deg}(f_D)=-\mathrm{deg}(f_{\partial Y})$,  then $f$ has a fixed point.
\end{thm}

\begin{proof} In the proof of Theorem \ref{ball}, we have
indicated that the condition $f^{-1}(\partial
X-Y)=\stackrel{\circ}Y$ is equivalent to $f(\stackrel{\circ}Y)
\subset \partial X-Y$, $f(\partial Y) \subset
\partial Y$, and $f(X-\stackrel{\circ}Y) \cap (\partial
X-Y)=\emptyset$. Thus $f(D-E) \cap (\partial X-Y) \subset
f(X-\stackrel{\circ}Y) \cap (\partial X-Y)=\emptyset$, and hence
$f(D) \cap \partial X=f(E) \subset f(\stackrel{\circ}Y) \subset
\partial X-Y$. By Proposition \ref{deg-int}, we have $f(D-E) \subset
\stackrel{\circ}X$. Take a point $v \in \stackrel{\circ}E$ and put
$W_v=\stackrel{\circ}D \cup \stackrel{\circ}E$. Then $W_v$ is an
open neighborhood of $v$ in $X$, and $f(W_v) \subset X$. Therefore,
by Theorem \ref{ball}, $\mathrm{Fix}(f) \neq \emptyset$.
\end{proof}

\begin{thm} \label{injec} Let $X$ be an $n$-ball
in $\mathbb{R}^n$, $n \geq 2$, $Y$ be an $(n-1)$-ball in $\partial
X$, and $f: X \rightarrow \mathbb{R}^n$ be a continuous injection
such that $f(Y)=\partial X-\stackrel{\circ}Y$. Let $f_{\partial
Y}=f|\partial Y: \partial Y \rightarrow \partial Y$. If
$\mathrm{deg}(f)=-\mathrm{deg}(f_{\partial Y})$, then $f$ has a
fixed point.
\end{thm}

\begin{proof} Let $D$ be an $n$-ball in $X$ such
that $\partial D \cap \partial X$ is an $(n-1)$-ball in
$\stackrel{\circ}Y$. Since $f$ is an injection, $f$ is bijective on
$D$ and $f^{-1}(\partial X-Y)=\stackrel{\circ}Y$. Let $f_D=f|D: D \rightarrow \mathbb{R}^n$. By Lemma \ref{restrict},
we have $\mathrm{deg}(f)=\mathrm{deg}(f_D)$. Hence, by Theorem \ref{main2'}, $f$ has a fixed point.
\end{proof}

Though the following corollary is implied by the Brouwer Lemma: if $f: \mathbb{R}^2
\rightarrow \mathbb{R}^2$ is an orientation preserving
homeomorphism  and has a periodic point, then it has a fixed point \cite{Bro12b},
it is still interesting to give a proof using the theorems here.

\begin{cor} Let $Y$ be an arc in
$\mathbb{R}^2$, $\partial Y=\{u, v\}$, and $f: \mathbb{R}^2
\rightarrow \mathbb{R}^2$ be an orientation preserving homeomorphism
such that $f(u)=v$, $f(v)=u$, and $f(Y) \cap Y =\partial Y$. Then
$f$ has a fixed point.
\end{cor}

\begin{proof} Clearly, there is a disc $X$ in
$\mathbb{R}^2$ such that $\partial X=f(Y) \cup Y$. Let $f_X=f|X: X
\rightarrow \mathbb{R}^2$ and $f_{\partial Y}=f|\partial Y:
\partial Y \rightarrow \partial Y$. Then $\mathrm{deg}(f_X)=\mathrm{deg}(f)=1$ and
$\mathrm{deg}(f_{\partial Y})=-1$. Therefore, by Theorem \ref{injec}, $f$
has a fixed point.
\end{proof}

\subsection*{Acknowledgements}
Jiehua Mai and Fanping Zeng are supported by NNSF of China (Grant
No. 12261006) and Project of Guangxi First Class Disciplines of
Statistics (No. GJKY-2022-01); Enhui Shi is supported by NNSF of
China (Grant No. 12271388); Kesong Yan is supported by NNSF of China
(Grant No. 12171175).

\end{document}